\documentclass[12pt, reqno]{amsart}

\usepackage{amsthm, amsmath, amsfonts, amscd, amssymb, stmaryrd, mathrsfs}
\usepackage[breaklinks,colorlinks,citecolor=blue,linkcolor=blue,urlcolor=blue,hyperindex]{hyperref}
\usepackage{color}
\usepackage{tikz}
\usepackage{hyperref}
\usetikzlibrary{arrows} \usetikzlibrary{patterns}
\usepackage[all]{xy}
\usepackage[top=2.5cm, bottom=2.5cm, left=3cm, right=3cm]{geometry}
\usepackage{backref}
\newtheorem{thm}{Theorem}[section]
\newtheorem{cor}[thm]{Corollary}
\newtheorem{lem}[thm]{Lemma}
\newtheorem{prop}[thm]{Proposition}

\theoremstyle{definition}

\newcommand{\ZN}{\mathbb{Z}}

\newcommand{\CO}{\mathcal{O}}

\newcommand{\BP}{\mathbb{P}}

\newcommand{\NS}{\mathrm{N}^{1}}

\newcommand{\Nef}{\mathrm{Nef}}

\newcommand{\Eff}{\mathrm{Eff}}

\newcommand{\Amp}{\mathrm{Amp}}

\newcommand{\Tors}{\mathrm{Tors}}

\makeatletter
\@namedef{subjclassname@2020}{%
  \textup{2020} Mathematics Subject Classification}
\makeatother

\allowdisplaybreaks

\numberwithin{equation}{section}

\begin{document}

\title[Some results on  fake quadrics]{Some results on   fake  quadrics.}

\author[J.-Q. Yang]{Jianqiang Yang}

\footnote{2010 Mathematics Subject Classification. Primary  14J29; Secondary  14J25,   32J15.

  Key words and phrases. Fake quadrics, cone, effective divisor, ample divisor, bounded cohomology property.

   }

\address{School of Mathematical Sciences, Guangxi Minzu University, Nanning, 530006, P. R. China.}
\email{yangjq\_math@sina.com}

\date{}

\begin{abstract} In this paper, we obtain several results on fake quadrics. First, we describe the cones \(\operatorname{Nef}(S)\) and \(\overline{\operatorname{Eff}(S)}\), and then provide a criterion for determining whether a divisor on a fake quadric is   ample or effective, respectively. Second, we show that no fake quadric can be embedded into $\mathbb P^4$. Finally,  we establish the bounded cohomology property for fake quadrics.
\end{abstract}

\maketitle


\section{introduction}\label{sec:1}

A \emph{fake quadric} is a minimal projective surface of general type with the same Betti numbers as a smooth quadric surface but not isomorphic to it. Equivalently, it is a minimal surface of general type with numerical invariants
\[
q(S)=p_g(S)=0,\qquad K_S^2=8,
\]
where $q(S)=\dim H^1(S,\mathcal O_S)$ is the irregularity, $p_g(S)=\dim H^0(S,\omega_S)$ is the geometric genus, and $K_S$ is the canonical divisor (cf. \cite{BCP08,BCP11,CCCK25,C25,Dza14,DR14}). Since such a surface shares the same $\mathbb Q$-homology algebra as a quadric, it is also called a fake $\mathbb Q$-homology quadric.

The study of fake quadrics is motivated by the classification of fake projective planes, i.e., surfaces of general type with the same Betti numbers as $\mathbb P^2$. Yau proved that every fake projective plane is a quotient $\mathbb B_2/\Gamma$ of the complex $2$-ball by a cocompact torsion-free discrete subgroup $\Gamma \subset \operatorname{Aut}(\mathbb B_2)$ \cite{Yau77,Yau78}. Subsequently, Prasad and Yeung gave a complete classification of all fake projective planes \cite{PY07}. In contrast, the classification of fake quadrics remains largely open. In particular, the existence of a simply-connected fake quadric, first raised by Hirzebruch (cf. \cite[p.8]{BCP11} and \cite[p.780]{Hir87}), is still unresolved.

Let $S$ be a fake quadric surface. It is well-known that the canonical divisor $K_S$ is ample. 
In fact, for any reduced and irreducible curve $C$ on $S$, since $K_S$ is nef, 
$K_S \cdot C = 0$ if and only if $C$ is a $(-2)$-curve. 
Let $k$ be the number of $(-2)$-curves on $S$. 
By a result of Miyaoka \cite[Proposition~2.1.1]{BHPV04,Miy84}, we have 
$k \le 2(3c_2 - c_1^2)/9 = 4/9 < 1$, hence $k = 0$. 
Therefore $K_S$ is ample.

For a fake quadric $S$, the exponential exact sequence yields an isomorphism $\operatorname{Pic}(S)\cong H^2(S,\mathbb Z)$. Hence the N\'eron--Severi lattice $\operatorname{NS}(S)\cong H^2(S,\mathbb Z)/\operatorname{tors}$ is a unimodular indefinite lattice of rank $2$ and signature $0$ by Poincar\'e duality and the Hodge index theorem. Consequently, $\operatorname{NS}(S)$ is either even or odd, i.e.,
\[
\operatorname{NS}(S)\cong U:=\begin{pmatrix}0 & 1\\1 & 0\end{pmatrix}\quad\text{or}\quad \operatorname{NS}(S)\cong \langle 1\rangle\oplus\langle -1\rangle:=\begin{pmatrix}1 & 0\\0 & -1\end{pmatrix}.
\]
We say that a fake quadric is of \emph{even type} in the former case and of \emph{odd type} in the latter.

Fake quadrics of odd type constitute the main focus of the present paper (cf. \cite{CCCK25,C25}). Historically, little attention has been paid to fake quadrics of odd type, and it was even unclear whether such surfaces exist at all, until F. Catanese provided explicit examples in \cite{C25}. Compared with the even type case, fake quadrics of odd type are significantly more challenging to handle, mainly due to the fact that it is not known a priori whether such a surface admits negative curves.  The odd type case is therefore the main difficulty of this paper.

Despite substantial progress on fake quadrics (cf. \cite{CCCK25,C25,Dza14,DR14,FL22,LSV19}), a systematic treatment of divisor effectiveness and ampleness has been lacking. The present paper addresses this gap. Specifically, we establish the following result.

\begin{thm}\label{thm:cone}
Let $S$ be a fake quadric of even type. Then the following cone relations hold:
\begin{equation}\label{thm:even}
\operatorname{Nef}(S)=\overline{\operatorname{Eff}(S)}=\mathbb R_{\ge 0}[D_1]+\mathbb R_{\ge 0}[D_2],
\end{equation}
where $D_1,D_2\subset S$ are two distinct divisors with $D_i^2=0$ for $i=1,2$.

Let $S$ be a fake quadric of odd type. 
\begin{itemize}
    \item If $S$ contains no negative curve, then the cone relations are the same as in \eqref{thm:even}.
    \item If a negative curve $E$ exists on $S$, then
    \[
    \overline{\operatorname{Eff}(S)} = \mathbb{R}_{\ge 0}[D_1 - D_2] + \mathbb{R}_{\ge 0}[E]
    \]
    and
    \[
    \operatorname{Nef}(S) = \mathbb{R}_{\ge 0}[D_1 - D_2] + \mathbb{R}_{\ge 0}[N],
    \]
    where the divisors $D_1,D_2\subset S$ satisfy 
    \[
    D_1^2 = 1,\quad D_2^2 = -1,\quad D_1 \cdot D_2 = 0,
    \]
    and $N$ is a nef divisor such that $N \cdot E = 0$.
\end{itemize}
\end{thm}

Note that the quadric $\mathbb P^1 \times \mathbb P^1$ is a hypersurface in $\mathbb P^3$, while the blow-up of $\mathbb P^2$ at a point, i.e., $\mathbb P^2 \sharp \overline{\mathbb P^2}$, embeds as a subvariety of $\mathbb P^4$. This naturally raises the question: does there exist a fake quadric that admits a holomorphic embedding into $\mathbb P^4$? Applying our divisor criteria, we obtain the following negative answer.

\begin{thm}\label{thm:P4}
No fake quadric can be holomorphically embedded into $\mathbb P^4$.
\end{thm}

Recall that a smooth projective surface $S$ satisfies the \emph{bounded cohomology property} if there exists a positive constant $c_S$ such that
\[
h^1(\mathcal O_S(C))\le c_S\cdot h^0(\mathcal O_S(C))
\]
for every reduced and irreducible curve $C\subset S$. This property originated from discussions among Harbourne, Ro\'e, Cilberto, and Miranda (cf. \cite{BBCRDHJK12}). As an application of our divisor criteria, we prove the following theorem for fake quadrics.

\begin{thm}\label{thm:bound}
Let $S$ be a fake quadric. Then the bounded cohomology property holds in the following cases:
\begin{itemize}
    \item $S$ is of even type;
    \item $S$ is of odd type and contains a negative curve $E$;
    \item $S$ is of odd type, contains no negative curve, and is a Mori dream space.
\end{itemize}
\end{thm}

The paper is organized as follows. Section~2 treats fake quadrics of even type, while Section~3 is devoted to the odd case. The more tedious calculations in the proof of Theorem \ref{thm:P4} are deferred to Section~4.

\section*{Notation and conventions}

We work throughout over the  field of  complex numbers $\mathbb C$. By a surface we mean a smooth projective complex surface. Unless otherwise specified, a curve on a surface is a non-zero effective divisor. We use the following standard notations:

\textbullet\ \  $h^{i}(S,\CO_{S}(D))$: $=\dim H^{i}(S,\CO_{S}(D))$;

\textbullet\ \  $C.D$: the intersection number of  two divisors  $C$ and $D$;

\textbullet\ \ $C\sim D:$ the two divisors $C,D$ are linear  equivalence;

\textbullet\ \ $C \equiv D:$ the two divisors $C,D$ are numerical  equivalence;

\textbullet\ \ $C =_{hom} D:$ the two divisors $C,D$ are homology  equivalence;

\textbullet\ \  $\Nef(S)$: the nef cone of the  surface $S$;

\textbullet\ \  $\Amp(S)$: the ample cone of the  surface $S$.

\textbullet\ \  $\Eff(S)$: the effective cone of the  surface $S$.

It is noteworthy that for a fake quadric $S$, by exponent exact sequence, then 
 \[C\sim D \Leftrightarrow C =_{hom} D \quad \text{and} \quad (C=_{hom} D) / \Tors \Leftrightarrow C \equiv D . \]
 
\subsection*{Acknowledgements}
We are grateful to Yifan Chen, Cheng-Yong Du, and Song Yang for their helpful communications. We would also like to extend our sincere appreciation to Fabrizio Catanese for his valuable insights on the N\'eron–Severi lattice of fake quadrics and for his thoughtful advice on this manuscript, shared via email.

This work was partially supported by the National Natural Science Foundation of China (12061014), the Guangxi Natural Science Foundation (2025GXNSFAA069316), and the Guangxi Science and Technology Program (AA24010005).

\section{Case I. A fake quadric of  even type}\label{sec:2}

Let $S$ be a fake quadric of even type.
By the definition, its N\'{e}ron--Severi lattice
$$
 \NS(S)
 =
 \left(
\begin{array}{cc}
0 & 1  \\
1 & 0 \\
\end{array}
\right)
$$
and hence there exist two divisors $D_1$ and $D_2$ such that
$$
D_1^{2}=0, D_2^{2}=0 \textrm{ and } D_1.D_2=1.
$$
From now on, we shall fix these two generators $D_1$ and $D_2$ in the N\'{e}ron--Severi lattice $\NS(S)$ of a fake quadric $S$ of even type.
Without loss of generality, since $K_{S}^{2}=8$,
we may suppose $K_{S}\equiv 2D_1+2D_2$.  In fact, the intersection form being even is equivalent to the vanishing 
of the second Stiefel--Whitney class $w_2(X)=0$ on $H_2(S,\mathbb Z)$ (cf. \cite[Chapter~10.2A, p.~164]{FQ90}). Since $w_2(S) \equiv c_1(S) \pmod{2}$ for a complex surface, 
this implies $c_1(S) \equiv 0 \pmod{2}$.
For a divisor $D\equiv xD_1+yD_2$,
by the Riemann--Roch theorem,
we have
\begin{equation}\label{even-RR-chi}
\chi(\CO_{S}(D))=1+\frac{1}{2}D(D-K_S)=(x-1)(y-1).
\end{equation}

\subsection{Criterion for Divisors: Effectiveness and Ampleness}\label{subsec:even-criterion}

\begin{lem}\label{eff-even}
Let $S$ be a fake quadric of even type and let $C \subset S$ be a reduced and irreducible curve.
If $C \equiv xD_1 + yD_2$ for some $x, y \in \mathbb Z$,
then either $x \ge 0$ and $y > 0$, or $x > 0$ and $y \ge 0$.
\end{lem}

\begin{proof}
Since $K_S$ is ample, we have $K_S.C > 0$, hence $x + y > 0$.
The arithmetic genus formula gives
\begin{equation}\label{even-ari-gen}
p_a(C) = \frac{1}{2}(K_S \cdot C + C^2) + 1 = (x+1)(y+1) \ge 0.
\end{equation}
Thus only two cases are possible:
\begin{enumerate}
\item[(1)] $x = -1$ and $y \ge 2$, or $x \ge 2$ and $y = -1$;
\item[(2)] $x \ge 0$ and $y > 0$, or $x > 0$ and $y \ge 0$.
\end{enumerate}

We now rule out case (1). 
Suppose $C \equiv -D_1 + yD_2$ with $y \ge 2$ (or $C \equiv xD_1 - D_2$ with $x \ge 2$).
Then $p_a(C) = 0$, so $C$ is a smooth rational curve.

The log Miyaoka–Yau inequality \cite{Lan01,LM95}   $(K_S + C)^2 \le 3\bigl(e(S) - e(C)\bigr)$  forces $y \le 1$ (respectively $x \le 1$), a contradiction.
This completes the proof of Lemma \ref{eff-even}.
\end{proof}

Based on the above statement,  we begin to prove Theorem \ref{thm:cone} of even type.

\begin{proof}  [ Proof of Theorem \ref{thm:cone} of the even type.]
By Lemma \ref{eff-even}, if $D \equiv xD_1 + yD_2$ is effective, then $x \ge 0$ and $y \ge 0$. Hence
\[
\operatorname{Eff}(S) \subset \mathbb R_{\ge 0}[D_1] + \mathbb R_{\ge 0}[D_2].
\]
For $x \ge 0$ and $y \ge 0$, we have $D_1. D \ge 0$ and $D_2.D \ge 0$; consequently, both $D_1$ and $D_2$ are nef. Thus
\[
\mathbb R_{\ge 0}[D_1],\; \mathbb R_{\ge 0}[D_2] \subset \operatorname{Nef}(S).
\]
Note that $\operatorname{Eff}(S)$ and $\operatorname{Nef}(S)$ are convex cones. Together with the inclusion $\operatorname{Nef}(S) \subset \overline{\operatorname{Eff}(S)}$, we obtain
\[
\operatorname{Nef}(S) = \mathbb R_{\ge 0}[D_1] + \mathbb R_{\ge 0}[D_2]
\quad\text{and}\quad
\overline{\operatorname{Eff}(S)} = \mathbb R_{\ge 0}[D_1] + \mathbb R_{\ge 0}[D_2].
\] 
\end{proof}

By  Theorem \ref{thm:cone}, we conclude that 

\begin{cor}\label{coro:ample-even}
Let $S$ be a fake quadric of even type and  divisor $D \equiv  xD_1+yD_2$. Then   the divisor $D$ is ample if and only if $x > 0$ and $y > 0$. 
\end{cor}

From  Theorem \ref{thm:cone} and Formula (\ref{even-ari-gen}), we can immediately obtain that

\begin{cor}\label{coro:genus-even}
Let $S$ be a fake quadric of even type. Then for any reduced and irreducible curve $C$ on $S$,  it holds $C^2 \geq 0$ and the arithmetic genus $p_{a}(C)\geq 2$.
\end{cor}

As a further application of the theorem, we specify precisely when the divisor is effective. This result will be needed in the next section.

\begin{thm}\label{coro:effective}
Let $S$ be a fake quadric of even type and let $D$ be a divisor on $S$ with $D \equiv xD_1 + yD_2$ for some $x, y \in \mathbb Z$.
\begin{enumerate}
\item[(i)] If $x, y \ge 2$, then $h^0(S, \mathcal O_S(D)) > 0$, unless $D \sim K_S$.
\item[(ii)] If $x, y \ge 3$, then
\[
h^0(S, \mathcal O_S(D)) = \chi(S, \mathcal O_S(D)) \ge 4.
\]
\end{enumerate}
\end{thm}

\begin{proof}
(i)    Since $K_S - D \equiv (2 - x)D_1 + (2 - y)D_2$, the assumptions together with Theorem \ref{thm:cone} imply that $K_S - D$ is either non-effective   or numerically trivial. In the latter case, $D \equiv K_S$, and if moreover $D \not\sim K_S$ then $K_S - D$ is a non-zero torsion divisor. Hence $h^0(S, \mathcal O_S(K_S - D)) = 0$  unless $D \sim K_S$. Consequently, the Riemann–Roch theorem yields
\[
h^0(S, \mathcal O_S(D)) \ge \chi(\mathcal O_S(D)) = (x-1)(y-1) \ge 1.
\]

(ii) By Theorem \ref{thm:cone}, the assumption $x,y$ guarantees that $D - K_S$ is ample. The Kodaira vanishing theorem then implies $h^i(S, \mathcal O_S(D)) = 0$ for all $i > 0$. Therefore,
\[
h^0(S, \mathcal O_S(D)) = \chi(S, \mathcal O_S(D)) = (x-1)(y-1) \ge 4.
\]
This completes the proof.
\end{proof}

\subsection{Embedding fake quadrics of even type  in $\BP^{4}$}

Now we prove Theorem \ref{thm:P4} for even type.

\begin{proof}[Proof of Theorem \ref{thm:P4} of the even type.]
Let $S$ be a fake quadric of even type and suppose there exists an embedding $\iota: S \hookrightarrow \mathbb P^4$.
Then $D := \iota^* \mathcal O_{\mathbb P^4}(1)$ is a very ample divisor on $S$.
Write $D \equiv xD_1 + yD_2$ with $x, y \in \mathbb Z$.
By Corollary \ref{coro:ample-even}, we have $x > 0$ and $y > 0$.

A well-known numerical condition for a surface embedded in $\mathbb P^4$ (see \cite[Appendix A, Example 4.1.3]{Har77}) gives
\begin{equation}\label{double-point-formula}
d^2 - 10d - 5D.K_S - 2K_S^2 + 12 + 12p_a(S) = 0,
\end{equation}
where $d = D^2$ and $p_a(S) = 0$.
Substituting $D \equiv xD_1 + yD_2$ and $K_S \equiv 2D_1 + 2D_2$, we obtain
\begin{equation}\label{embed-eq1}
2x^2y^2 - 10xy - 5x - 5y - 2 = 0, \qquad x, y > 0.
\end{equation}

If $0 < x < 3$ or $0 < y < 3$, by Lemma \ref{lem:1}, equation \eqref{embed-eq1} has no solutions.

By \cite[Chapter V, Corollary 2.12]{Zak93}, we have $h^0(S, \mathcal O_S(D)) < \binom{4}{2} = 6$.
Together with Theorem \ref{coro:effective}(ii), this yields
\begin{equation}\label{bound-eq1}
(x-1)(y-1) = h^0(S, \mathcal O_S(D)) < 6.
\end{equation}

For $x \ge 3$ and $y \ge 3$, by Lemma \ref{lem:2},  equations \eqref{embed-eq1} and \eqref{bound-eq1} cannot be satisfied simultaneously.

Thus no such embedding exists, completing the proof.
\end{proof}

According to  the formula \eqref{even-RR-chi},   the above theorem immediately yields the following Corollary.

\begin{cor}\label{coro:very 1}
Let $S$ be a fake quadric of even type. The the divisor $D \equiv 3D_1+3 D_2$ is not very ample.
\end{cor}

\begin{proof} By Lemma \ref{coro:effective}, we have that \[h^0(S, \mathcal O_S(D)) = \chi(S, \mathcal O_S(D))=(3-1)(3-1)=4.\]  If such divisor $D$ is very ample, then $\tau_{\rvert D \rvert}: S \to \BP^3$ is an embedding. This contradicts the Theorem \ref{thm:P4}.
\end{proof}

\subsection{The bounded cohomology property for the  fake quadrics  of  even type}

In this subsection, a curve always means a reduced and irreducible curve on the surface.
The following lemma will be needed.

\begin{lem}\label{lem:xy2}
Let $S$ be a fake quadric of even type and let $C \equiv xD_1 + yD_2$ be a curve on $S$ with $x, y \in \mathbb Z$.
If $x, y \ge 2$, or $x = 1$, or $y = 1$, then $C$ satisfies the bounded cohomology property.
\end{lem}

\begin{proof} 
Note that $h^0(\CO_S(K_S))=h^2(\CO_S)=0$.  This implies that $K_S - C$ is not effective. Serre duality gives
\[
h^2(S, \mathcal O_S(C)) = h^0(S, \mathcal O_S(K_S - C)) = 0.
\]
Hence the Riemann–Roch theorem yields
\[
h^0(S, \mathcal O_S(C)) - h^1(S, \mathcal O_S(C)) = \chi(\mathcal O_S(C)) = (x-1)(y-1).
\]

If $x, y \ge 2$, then $h^1(S, \mathcal O_S(C)) < h^0(S, \mathcal O_S(C))$.
If $x = 1$ or $y = 1$, then $h^1(S, \mathcal O_S(C)) = h^0(S, \mathcal O_S(C))$.
In summary, we have
\[
h^1(S, \mathcal O_S(C)) \le h^0(S, \mathcal O_S(C)).
\]
\end{proof}

Before we completely resolve the theorem, let us first look at a useful conclusion.

\begin{prop}\label{prop:D2=0}
\cite[Proposition 1.3]{AL11}
Let $S$ be a smooth projective surface with $q(S) = 0$ and let $D$ be a non-trivial effective divisor on $S$ such that $|D|$ has no fixed components and $D^2 = 0$.
Then $D = aC$ for some $a \in \mathbb Z_{>0}$, where $C$ is smooth and irreducible, $h^0(S, \mathcal O_S(C)) = 2$, and $H^0(S, \mathcal O_S(D)) \cong \operatorname{Sym}^a H^0(S, \mathcal O_S(C))$.
\end{prop}

With the above preparations, we now prove the Theorem \ref{thm:bound}.

\begin{proof}[Proof of Theorem \ref{thm:bound} of the even type]
In light of Lemma \ref{lem:xy2}, it remains to consider the case where $x = 0$ or $y = 0$.
Without loss of generality, assume $C \equiv x_0 D_1$ for some $x_0 > 0$.  The case $C \equiv x_0 D_2$ is similar.
Since $C$ is reduced and irreducible and $C^2 = 0$, by Proposition \ref{prop:D2=0},  we have $h^0(S, \mathcal O_S(C)) \in \{1, 2\}$.

\emph{Case 1:} $h^0(S, \mathcal O_S(C)) = 1$.
From
\[
h^0(S, \mathcal O_S(C)) - h^1(S, \mathcal O_S(C)) = \chi(\mathcal O_S(C)) = (x-1)(y-1) = 1-x_0
\]
and $h^0(S, \mathcal O_S(C)) = 1$, we obtain $h^1(S, \mathcal O_S(C)) = x_0$.
Thus
\[
h^1(S, \mathcal O_S(C)) = x_0 \cdot h^0(S, \mathcal O_S(C)).
\] Note that if $C_1 \equiv C_2 \equiv C$, then $C_1$ and $C_2$ differ by a torsion divisor. Therefore, there are only finitely many values of $x_0$; let $c_1$ be the maximum of $x_0$ over these values. Then
\[
h^1(S, \mathcal O_S(C)) \le c_1 \cdot h^0(S, \mathcal O_S(C)).
\]

\emph{Case 2:} $h^0(S, \mathcal O_S(C)) = 2$.
A similar computation gives $h^1(S, \mathcal O_S(C)) = x_0 + 1$, so
\[
h^1(S, \mathcal O_S(C)) = \frac{x_0 + 1}{2} \cdot h^0(S, \mathcal O_S(C)).
\]
By Proposition \ref{prop:D2=0},  only finitely many values of $x_0$ can occur; let $c_2$ be the maximum of $\frac{x_0+1}{2}$ over these values. Then
\[
h^1(S, \mathcal O_S(C)) \le c_2 \cdot h^0(S, \mathcal O_S(C)).
\]

Combining both cases, we conclude that the bounded cohomology property holds for all curves on $S$.
\end{proof}

\section{Case II. A fake quadric of  odd type}\label{sec:3}

Let $S$ be a fake quadric of odd type.
According to the definition, the N\'{e}ron--Severi lattice
$$
 \NS(S)
 =
 \left(
\begin{array}{cc}
1 & 0  \\
0 & -1 \\
\end{array}
\right)
$$
and thus there are two divisors $D_1$ and $D_2$ with intersection numbers
$$
D_1^{2}=1,\; D_2^{2}=-1 \;\text{ and }\; D_1.D_2=0.
$$
From now on, we will fix these two generators $D_1$ and $D_2$ of the N\'{e}ron--Severi lattice $\NS(S)$.
  Recall that the intersection form of $S$ is odd if and only if 
the second Stiefel--Whitney class $w_2(S)$ is nonzero. This menas that  the canonical divisor $K_{S}$ is a primitive divisor in $ \NS(S)$.
Without loss of generality, since $K_{S}^{2}=8$, we may assume $K_{S}\equiv 3D_1-D_2$.
For a divisor $D\equiv xD_1+yD_2$,
by the Riemann--Roch theorem,
we have
\begin{equation}\label{odd-RR-chi}
\chi(\CO_{S}(D)) = \frac{1}{2}(x+y-1)(x-y-2).
\end{equation}

\subsection{Criterion for Divisor: Effectiveness and Ampleness}

\begin{lem}\label{eff-odd}
Let $S$ be a fake quadric of odd type and $C$ be a reduced and irreducible curve on $S$.
If $C\equiv xD_1+yD_2$ for some $x, y\in \ZN$,
then $x\geq 0$ and $-x-1<y\leq x+1$.
\end{lem}

\begin{proof}
Since $K_{S}$ is ample, we have $K_{S}.C>0$ and thus $3x+y>0$.
By the arithmetic genus formula, we get
\begin{equation}\label{odd-ari-gen}
p_{a}(C)=\frac{1}{2}(x+y+1)(x-y+2)\geq 0.
\end{equation}
Then we have only two possible cases:
\begin{enumerate}
\item[(1)] $x=0$ and $y=2$, or $x>0$ and $y=-(x+1)$, or $x>0$ and $y=x+2$, or
\item[(2)] $x\geq 0$ and $-(x+1)< y\leq x+1$.
\end{enumerate}
We will rule out case $(1)$.
The idea is the same as in the proof of Lemma \ref{eff-even}.
Suppose $C\equiv 2D_2$ or $C\equiv xD_1-(x+1)D_2$ or $C\equiv xD_1+(x+2)D_2$ with $x>0$.
Then the arithmetic genus $p_{a}(C)=0$ and thus $C$ is a smooth rational curve.
The log Miyaoka-Yau inequality $(K_X+C)^2 \leq 3(e(S)-e(C))$ yields the following constraint: 
\[(x+3)^2-(y-1)^2 \leq 6.\] Consequently, case (1) does not satisfy this inequality.
Thus the proof of Lemma \ref{eff-odd} is complete.
\end{proof}

The above result shows that, unlike the even type, the odd type may admit negative curves.

\begin{lem} \label{lem:negative}  Let $S$ be a fake quadric whose N\'eron--Severi lattice is of odd type. If an integral negative curve $E$ on $S$ exists, then it is unique and satisfies $E \equiv x_0 D_1 + (x_0 + 1) D_2$ for some integer $x_0 \ge 0$.
\end{lem}

\begin{proof} According to Lemma~\ref{eff-odd}, any integral negative curve on a fake quadric $S$ is numerically equivalent to $x D_1 + (x+1) D_2$ for some $x \in \mathbb{Z}_{\geq 0}$.

Let $x_0 \in \mathbb{Z}_{\geq 0}$ be the minimal integer such that $x_0 D_1 + (x_0+1) D_2$ is numerically equivalent to an effective divisor. Define the curve $E \equiv x_0 D_1 + (x_0+1) D_2$. Since $E^2 < 0$, we have $\dim |E| = 0$.

Now consider any divisor $D \equiv x D_1 + (x+1) D_2$ with $x > x_0$. If $D$ were effective and some element of $|D|$ were reduced and irreducible, then necessarily $\dim |D| = 0$ (otherwise a general member would be smooth and irreducible, contradicting the minimality of $x_0$). In this situation, $D - E$ would not be effective. Moreover, $\deg(D\vert_E) < 0$ implies $h^0(E, \mathcal{O}_E(D)) = 0$. From the exact sequence
\[
0 \to H^0(S, \mathcal{O}_S(D-E)) \to H^0(S, \mathcal{O}_S(D)) \to H^0(E, \mathcal{O}_E(D)) \to \cdots
\]
and $h^0(S, \mathcal{O}_S(D-E)) = 0$, we would obtain $h^0(S, \mathcal{O}_S(D)) = 0$, a contradiction. Thus any negative curve must be numerically equivalent to $E$.

Suppose $E' \equiv E$ is another negative curve. Then $E$ and $E'$ differ by a torsion divisor, so there exists $m \in \mathbb{Z}_{>0}$ such that $mE \sim mE'$. Consequently, $h^0(S, \mathcal{O}_S(mE)) \ge 2$. But $E^2 < 0$, and using the exact sequence
\[
0 \to H^0(S, \mathcal{O}_S((k-1)E)) \to H^0(S, \mathcal{O}_S(kE)) \to H^0(E, \mathcal{O}_E(kE)) \to \cdots,
\]
we obtain inductively
\[
h^0(S, \mathcal{O}_S(mE)) = h^0(S, \mathcal{O}_S((m-1)E)) = \cdots = h^0(S, \mathcal{O}_S(E)) = 1,
\]
a contradiction. Therefore, there can be at most one integral negative curve on $S$.
\end{proof}

  We shall now proceed to further investigate and refine this result.

 \begin{lem} \label{lem:nef cone} Let $S$ be a fake quadric of odd type. If a negative curve 
\[
E \equiv x_0 D_1 + (x_0 + 1) D_2
\]
exists on $S$ for some integer $x_0 \ge 0$, then the cone structures are given by
\[
\overline{\operatorname{Eff}(S)} = \mathbb{R}_{\ge 0}[D_1 - D_2] + \mathbb{R}_{\ge 0}[E]
\]
and
\[
\operatorname{Nef}(S) = \mathbb{R}_{\ge 0}[D_1 - D_2] + \mathbb{R}_{\ge 0}[N],
\]
where $N \equiv (x_0 + 1)D_1 + x_0 D_2$.
\end{lem}

 \begin{figure}[htbp]
\centering
\begin{tikzpicture}[scale=0.5]
\draw[->] (-5,0) -- (6,0) node[right] {$D_1$}; 
\draw[->] (0,-5) -- (0,6) node[above] {$D_2$}; 

\foreach \x in {-5,-4,-3,-2,-1,1,2,3,4,5} \draw (\x,0) -- (\x,-0.1) node[below] {\x};
\foreach \y in {-5,-4,-3,-2,-1,1,2,3,4,5} \draw (0,\y) -- (-0.1,\y) node[left] {\y};

\draw[domain=0:5,smooth,variable=\x,blue] plot ({\x},{-\x}) node[right] {$\mathbb R_{\geq 0} [D_1-D_2]$}; 
\draw[domain=0:5,smooth,variable=\x,red,dashed] plot ({\x},{\x}) node[right] {$\mathbb R_{\geq 0} [D_1+D_2]$}; 
\draw[domain=0:5,smooth,variable=\x,blue] plot ({\x},{3/2*\x}) node[right] {$\mathbb R_{\geq 0} [E]$}; 
\draw[domain=0:5,smooth,variable=\x,blue]  plot ({\x},{2/3*\x}) node[right] {$\mathbb R_{\geq 0}[N]$}; 

\foreach \x in {-5,...,5} {
    \foreach \y in {-5,...,5} {
        \fill (\x,\y) circle (2pt);
    }
}
\end{tikzpicture}
\caption{Cones of a fake quadric of odd type}
\label{fig:cone1}
\end{figure}

\begin{proof} From Lemma~\ref{eff-odd} and Lemma~\ref{lem:negative}, we obtain the inclusions
\[
\operatorname{Nef}(S) \subseteq \overline{\operatorname{Eff}(S)} \subseteq \mathbb{R}_{\ge 0}[D_1 - D_2] + \mathbb{R}_{\ge 0}[E].
\]

 Now suppose a divisor satisfies
\[
D \in \mathbb{R}_{\ge 0}[N] + \mathbb{R}_{\ge 0}[E].
\]
Then we have \(D \cdot E = m E^2 \le 0\) for some $m \geq 0$. This implies that the cone
\(\mathbb{R}_{\ge 0}[E] + \mathbb{R}_{\ge 0}[N]\) contains no ample divisor.

Since \(\overline{\operatorname{Amp}(S)} = \operatorname{Nef}(S)\), it follows that
\[
\operatorname{Nef}(S) \subseteq \mathbb{R}_{\ge 0}[D_1 - D_2] + \mathbb{R}_{\ge 0}[N].
\]

Observe that any divisor \(C \in \mathbb{R}_{\ge 0}[D_1 - D_2] + \mathbb{R}_{\ge 0}[E]\) satisfies
\(C \cdot (D_1 - D_2) \ge 0\); hence \(D_1 - D_2\) is a nef divisor.  It follows from the lemma~\ref{eff-odd} that $N$ is nef.  Consequently,
\[
\mathbb{R}_{\ge 0}[D_1 - D_2] \subset \operatorname{Nef}(S), \quad
\mathbb{R}_{\ge 0}[N] \subset \operatorname{Nef}(S),
\]
and therefore
\[
\operatorname{Nef}(S) = \mathbb{R}_{\ge 0}[D_1 - D_2] + \mathbb{R}_{\ge 0}[N].
\]

Finally, because \(\overline{\operatorname{Eff}(S)}\) is the dual cone of \(\operatorname{Nef}(S)\), we obtain
\[
\overline{\operatorname{Eff}(S)} = \mathbb{R}_{\ge 0}[D_1 - D_2] + \mathbb{R}_{\ge 0}[E].
\]
\end{proof}

\begin{lem}\label{lem:nef cone2}
Let $S$ be a fake quadric of odd type. 
Assume that no negative curve exists on $S$. Then The cones satisfy
    \[
    \operatorname{Nef}(S) = \overline{\operatorname{Eff}(S)} = \mathbb{R}_{\ge 0}[D_1 - D_2] + \mathbb{R}_{\ge 0}[D_1 + D_2].
    \]
\end{lem}

\begin{proof}
 By the assumption, Lemma~\ref{eff-odd} immediately implies that if $D\equiv xD_1+yD_2$ is effective, then $x \ge |y|$. This yields the inclusion
\[
\operatorname{Eff}(S) \subset \mathbb{R}_{\ge 0}[D_1 - D_2] + \mathbb{R}_{\ge 0}[D_1 + D_2].
\]

Since
\[
(D_1 - D_2)^2 = (D_1 + D_2)^2 = 0 \quad \text{and} \quad (D_1 - D_2) \cdot D \ge 0,\;\; (D_1 + D_2) \cdot D \ge 0
\]
whenever $x \ge |y|$, it follows that $D_1 - D_2$ and $D_1 + D_2$ are nef divisors. Consequently,
\[
\mathbb{R}_{\ge 0}[D_1 - D_2],\; \mathbb{R}_{\ge 0}[D_1 + D_2] \subset \operatorname{Nef}(S).
\]
And then $\mathbb{R}_{\ge 0}[D_1 - D_2] + \mathbb{R}_{\ge 0}[D_1 + D_2] \subset \operatorname{Nef}(S)$
Together with the inclusion $\operatorname{Nef}(S) \subset \overline{\operatorname{Eff}(S)}$, we obtain
\[
\operatorname{Nef}(S)=\overline{\operatorname{Eff}(S)}  = \mathbb{R}_{\ge 0}[D_1 - D_2] + \mathbb{R}_{\ge 0}[D_1 + D_2]
\]
\end{proof}

\begin{proof} [ Proof of Theorem \ref{thm:cone} of the odd type.]
The combination of Lemma \ref{lem:nef cone} and Lemma \ref{lem:nef cone2} completes the proof of Theorem \ref{thm:cone} for the odd type. 
\end{proof}

Additionally, together with the arithmetic genus Formula \eqref{odd-ari-gen}, we have the following corollary.

\begin{cor}\label{coro:genus-odd}
Let $S$ be a fake quadric of odd type. There is no smooth rational curve on $S$.
\end{cor}

As another consequence of Theorem \ref{thm:cone}, we can obtain the following result.

\begin{cor}\label{coro:genus-odd}
Let $S$ be a fake quadric of odd type. If the divisor $D=xD_1+yD_2$ is ample, then $x> \lvert y \rvert$.
\end{cor}

\begin{proof} Note that $\mathring{\Nef}(S)=\Amp(S),$ where  $\mathring{\Nef}(S)$ is the set of interior points of the nef cone $\Nef(S).$  By Theorem \ref{thm:cone} (or Lemma \ref{lem:nef cone} and Lemma \ref{lem:nef cone2}), the result holds. 
\end{proof}

Based on this numerical condition, we can further provide a sufficient criterion for the effectiveness of a divisor, as stated in the following theorem.

\begin{thm}\label{lem:odd}
Let $S$ be a fake quadric of odd type and $D\equiv x D_1 + y D_2$ a divisor on $S$.  If $x + y > 1$ and $x - y > 2$, then $h^0(S, \mathcal{O}_S(D)) \geq \chi(\mathcal{O}_S(D)) > 0$ unless $D \sim K_S$.
\end{thm}

\begin{proof}
 If $K_S - D \equiv (3 - x)D_1 - (y + 1)D_2$ is effective, then by Theorem \ref{thm:cone} or Lemma \ref{eff-odd} we have
\[
x - 3 \le -(y + 1) \le (3 - x) + 1,
\]
which implies $x - y \le 5$ and $x + y \le 2$. Consequently, $x \le 3$.

Given $x + y > 1$ and $x - y > 2$, we obtain $x \ge 2$. Thus we only need to consider $x = 2$ or $x = 3$. The conditions $x + y > 1$ and $x + y \le 2$ force $x + y = 2$.

\begin{itemize}
\item If $x = 2$, then $y = 0$, which contradicts $x - y > 2$.
\item If $x = 3$, then $y = -1$, and hence $D \equiv K_S$. In this case, $K_S - D$ is a torsion divisor, again a contradiction.
\end{itemize}

Therefore, we conclude that $h^0(S, \mathcal{O}_S(K_S - D)) = 0$.

Hence, by the Riemann--Roch theorem,
\[
h^0(S, \mathcal{O}_S(D)) \ge \chi(S, \mathcal{O}_S(D)) = \frac{1}{2}(x + y - 1)(x - y - 2) \ge 1.
\] 
\end{proof}

\subsection{Embedding fake quadrics of odd type in $\BP^{4}$}

In this subsection, we continue our investigation of fake quadrics that cannot be embedded into $\BP^{4}$ and complete the remaining proof of Theorem \ref{thm:P4}.

\begin{proof} [ Proof of Theorem \ref{thm:P4} of the odd type.]

Let $S$ be a fake quadric of odd type and suppose there exists an embedding $\iota: S \hookrightarrow \mathbb P^4$.
Using the same notation as in the even case, assume that $D \equiv x D_1 + y D_2$ is an ample divisor. Then equation \eqref{double-point-formula} implies that $x$ and $y$ must satisfy
\[
(x^2 - y^2)^2 - 10(x^2 - y^2) - 5(3x + y) - 4 = 0.
\]
From Corollary \ref{coro:genus-odd}, we have $x > |y|$.

A suitable rearrangement of the above equation yields
\begin{equation}\label{embed-eq2}
(x^2 - y^2 - 5)^2 = 5(3x + y) + 29, \quad \text{for } x > |y|.
\end{equation}

If either $x + y \le 1$ or $x - y \le 2$, then Lemma \ref{lem:3} implies that equation \eqref{embed-eq2} has no solution.

Similar to the even type case, we have
\[
h^0\bigl(S, \mathcal{O}_S(D)\bigr) < \binom{4}{2} = 6.
\]

If $x + y > 1$ and $x - y > 2$, then by Theorem \ref{lem:odd}, unless $D \sim K_S$, we obtain
\[
\frac{1}{2} (x + y - 1)(x - y - 2) = \chi\bigl(S, \mathcal{O}_S(D)\bigr) \le h^0\bigl(S, \mathcal{O}_S(D)\bigr) < 6.
\]
Consequently,
\begin{equation}\label{bound-eq2}
(x + y - 1)(x - y - 2) \le 10.
\end{equation}
By Lemma \ref{lem:4}, none of the integer pairs satisfying \eqref{bound-eq2} fulfill equation \eqref{embed-eq2}.

This completes the proof.
\end{proof}

According to  the formula \eqref{odd-RR-chi},   the above theorem immediately yields the following corollary.

\begin{cor}\label{coro:very 2}
Let $S$ be a fake quadric of odd type. Then the divisors $D \equiv 4D_1- D_2$ and $D \equiv 5D_1- 2D_2$ are not very ample.
\end{cor}

\begin{proof} We mainly consider $D \equiv 4D_1 - D_2$; the same method can be applied to prove the case $D \equiv 5D_1 - 2D_2$. Suppose that $D=K_S+L$, then $L\equiv D_1$. According to Theorem \ref{thm:cone}, $L$ is nef and big. By Kawamata-Viehweg Theorem, we obtain that
\[h^0(S,\CO_S(D))=\chi(\CO_S(D))= \frac{1}{2}(x+y-1)(x-y-2)=3.\]  If such divisor $D$ is very ample, then $\tau_{\rvert D \rvert}: S \to \BP^2$ is an embedding. This contradicts the Theorem \ref{thm:P4}. 
\end{proof}

\subsection{The bounded cohomology property for  fake quadrics of  odd type}

In this subsection, a curve also means a reduced and irreducible curve on a surface.
Using the same argument as the even type (Sec.2.4. Lemma  \ref{lem:xy2}), we can also prove the following result.

\begin{prop}\label{lem:x-y>2}  If $S$ is a fake quadric of  odd type and a curve $C\equiv xH+yF,x,y \in \mathbb Z$,  on $S$,
   then $C$ satisfies the bounded cohomology property for  
\[ x+y> 1 \quad \text{and} \quad x-y> 2,   \quad \text{or} \quad x+y=1, \quad \text{or} \quad  x-y=2,\quad \text{or} \quad x=\pm y.\]
\end{prop}

\begin{proof}   Recall that $$\chi(\CO_S(C))=h^0(\CO_S(C))-h^1(\CO_S(C)).$$ We consider four cases.

\noindent\textbf{Case 1: $x + y > 1$ and $x - y > 2$.} 
Since
\[
h^0(\mathcal{O}_S(C)) - h^1(\mathcal{O}_S(C)) = \chi(\mathcal{O}_S(C)) = \frac{1}{2}(x+y-1)(x-y-2) > 0,
\]
we obtain $h^1(\mathcal{O}_S(C)) < h^0(\mathcal{O}_S(C))$.

\noindent\textbf{Case 2: $x + y = 1$ or $x - y = 2$.} 
Since
\[
h^0(\mathcal{O}_S(C)) - h^1(\mathcal{O}_S(C)) = \chi(\mathcal{O}_S(C)) = \frac{1}{2}(x+y-1)(x-y-2) = 0,
\]
we obtain $h^1(\mathcal{O}_S(C)) = h^0(\mathcal{O}_S(C))$.

\noindent\textbf{Case 3: $x = -y$.} 
Since $C^2 = 0$, by Proposition~\ref{prop:D2=0}, we have $h^0(S, \mathcal{O}_S(C)) \in \{1, 2\}$. 
Using the same argument as for the even type (Section~2.3), this case can be proved.

\noindent\textbf{Case 4: $x = y$.} 
If no negative curve exists on $S$, the argument is the same as in Case~3. 
If there exists a negative curve $E \equiv x_0 D_1 + (x_0 + 1)D_2$ on $S$, then since $C$ is reduced and irreducible, either $C$ is a negative curve, or it lies in the cone $\mathbb{R}_{\ge 0}[D_1 - D_2] + \mathbb{R}_{\ge 0}[N]$, implying $x \neq y$.
\end{proof}

Next, we only need to verify whether the case $x - y = 1$ satisfies the bounded cohomology property. If a negative curve  exists on $S$, we have the following result.

\begin{lem}  Let $S$ be a fake quadric of odd type.  If a negative curve $E \equiv x_0 D_1 + (x_0 + 1)D_2$ exists on $S$, then the bounded cohomology property holds.
\end{lem}

\begin{proof} From the above discussion, we restrict our attention to curves of the form $C \equiv (x+1)D_1 + xD_2$. Consider the ray $\mathbb R_{\ge 0}[N]$ intersecting the line $(x+1)D_1 + xD_2$ at $(x_0 + 1, x_0)$. When $x > x_0$, we have that $C$ is neither a negative curve nor contained in the cone $\mathbb{R}_{\ge 0}[D_1 - D_2] + \mathbb{R}_{\ge 0}[N]$; hence $C$ is not irreducible.  When $x \le x_0$, there are only finitely many such curves $C$, and the conclusion  holds trivially.
\end{proof}

\begin{figure}[htbp]
\centering
\begin{tikzpicture}[scale=0.5]
\draw[->] (-5,0) -- (9,0) node[right] {$D_1$};
\draw[->] (0,-5) -- (0,10) node[above] {$D_2$};

\foreach \x in {-5,-4,-3,-2,-1,1,2,3,4,5,6,7,8} \draw (\x,0) -- (\x,-0.1) node[below] {\x};
\foreach \y in {-5,-4,-3,-2,-1,1,2,3,4,5,6,7,8,9} \draw (0,\y) -- (-0.1,\y) node[left] {\y};

\draw[domain=0:8,smooth,variable=\x,blue] plot ({\x},{-\x}) node[right] {$\mathbb R_{\geq 0} [D_1-D_2]$};
\draw[domain=0:8,smooth,variable=\x,red,dashed] plot ({\x},{\x}) node[right] at (6,9) {$\mathbb R_{\geq 0} [D_1+D_2]$};
\draw[domain=0:8,smooth,variable=\x,blue] plot ({\x},{3/2*\x}) node[right] {$\mathbb R_{\geq 0} [E]$};
\draw[domain=0:8,smooth,variable=\x,blue] plot ({\x},{2/3*\x}) node[right] {$\mathbb R_{\geq 0}[N]$};

\draw[domain=-4:8,smooth,variable=\x,green!70!black,thick] plot ({\x},{1-\x}) node[above right,green!70!black] {$x+y=1$};
\draw[domain=-3:8,smooth,variable=\x,green!70!black,thick] plot ({\x},{\x-2}) node[above right,green!70!black] at (8,6) {$x-y=2$};
\draw[domain=-4:8,smooth,variable=\x,orange!70!black,thick] plot ({\x},{\x-1}) node[above right,orange!70!black] {$x-y=1$};

\fill[black] (3,2) circle (3pt); 
\node[black, above right] at (3,2) {$(x_0+1, x_0)$};

\foreach \x in {-5,-4,...,8} {
    \foreach \y in {-5,-4,...,9} {
        \fill (\x,\y) circle (2pt);
    }
}
\end{tikzpicture}
\caption{Cones of a fake quadric of odd type with additional lines}
\label{fig:cone2}
\end{figure}

If $C\equiv (x+1)D_1 + xD_2$, we always have
\[
h^{0}(S,\mathcal{O}_{S}(C)) - h^{1}(S,\mathcal{O}_{S}(C)) = \chi(\mathcal{O}_{S}(C)) = \frac{1}{2}(x+1+x-1)(x+1-x-2) = -x.
\]
However, when a negative curve exists, this formula is not needed, as the problem can be handled as described above. In contrast, when no negative curve exists, we are forced to rely on this formula, which gives a negative value and makes it difficult to determine the bounded cohomology property for $C \equiv (x+1)D_1 + xD_2$. Therefore, we only consider the case when $S$ is a Mori dream space.

\begin{proof}[Proof of Theorem \ref{thm:bound} of the odd type]
Under the above discussion, we still only consider the curve $C \equiv (x+1)D_1 + xD_2$. Let $C_1 \equiv D_1 - D_2$, $C_2 \equiv D_1 + D_2$, and assume that no negative curve exists on $S$.

If $S$ is a Mori dream space,  then by \cite{KL19},  $C_i$ is semiample for $i=1,2$. Since $C_i^2=0$, by Proposition \ref{prop:D2=0}, the Iitaka dimension satisfies $\kappa(S, C_i) = 1$.

Let $x_0$ be the minimal integer such that $h^0(S, \mathcal{O}_S((x+1)D_1 + xD_2)) > 0$ for all $x \ge x_0$. Define
\[
C_0 \equiv (x_0+1)D_1 + x_0D_2 \quad \text{and} \quad x = x_0 + m,\quad m \gg 0.
\]
According to \cite[2.1.38]{LA04}, we have $h^0(S, \mathcal{O}_S(mC_i)) \ge c m$ for some constant $c > 0$. Consequently,
\[
h^0(S, \mathcal{O}_S(C)) = h^0(S, \mathcal{O}_S(C_0 + mC_2)) \ge h^0(S, \mathcal{O}_S(mC_2)) \ge c m.
\]
Thus $c^{-1} h^0(S, \mathcal{O}_S(C)) \ge m$.

Recall that
\[
h^0(S, \mathcal{O}_S(C)) + x = h^1(S, \mathcal{O}_S(C)).
\]
Therefore,
\[
(c^{-1} + 1) h^0(S, \mathcal{O}_S(C)) \ge m + h^0(S, \mathcal{O}_S(C)) = h^1(S, \mathcal{O}_S(C)) - x_0.
\]
Since $h^0(S, \mathcal{O}_S(C)) > x_0$ for sufficiently large $m$, we obtain
\[
(c^{-1} + 2) h^0(S, \mathcal{O}_S(C)) > (c^{-1} + 1) h^0(S, \mathcal{O}_S(C)) + x_0 \ge h^1(S, \mathcal{O}_S(C)).
\]
This implies that the bounded cohomology property holds.
\end{proof}

Moreover, the above proof implies the following corollary.

\begin{cor}
Let $S$ be a fake quadric of odd type and no negative curve exists on $S$. If $D_1+D_2$ is semiample,  then the bounded cohomology property is satisfied.
\end{cor}

\appendix
\section{Appendix}

In the proof of Theorem \ref{thm:P4}, some tedious calculations are required. We present the details in this chapter.

\begin{lem}\label{lem:1}
There are no integer solutions to the system of inequalities \eqref{embed-eq1} with $0 < x < 3$, nor with $0 < y < 3$.
\end{lem}

\begin{proof}
Since $x$ and $y$ are positive integers, the condition $0 < x < 3$ implies $x = 1$ or $x = 2$.  
Similarly, $0 < y < 3$ implies $y = 1$ or $y = 2$.

We examine each case.

\noindent\textbf{Case 1: $x = 1$.}  
Substituting $x = 1$ into the equation gives
\[
2 \cdot 1^2 \cdot y^2 - 10 \cdot 1 \cdot y - 5 \cdot 1 - 5y - 2 = 0,
\]
which simplifies to
\[
2y^2 - 15y - 7 = 0.
\]
Here $a = 2$, $b = -15$, $c = -7$. The discriminant is
\[
\Delta = b^2 - 4ac = (-15)^2 - 4 \cdot 2 \cdot (-7) = 225 + 56 = 281.
\]
Since $281$ is not a perfect square, the quadratic has no integer roots. Hence no integer $y > 0$ satisfies the equation when $x = 1$.

\noindent\textbf{Case 2: $x = 2$.}  
Substituting $x = 2$ gives
\[
2 \cdot 4 \cdot y^2 - 10 \cdot 2 \cdot y - 5 \cdot 2 - 5y - 2 = 0,
\]
which simplifies to
\[
8y^2 - 25y - 12 = 0.
\]
Here $a = 8$, $b = -25$, $c = -12$. The discriminant is
\[
\Delta = (-25)^2 - 4 \cdot 8 \cdot (-12) = 625 + 384 = 1009.
\]
Since $1009$ is not a perfect square, there is no integer $y > 0$ satisfying the equation when $x = 2$.

\noindent\textbf{Case 3: $y = 1$.}  
By symmetry of the equation in $x$ and $y$, substituting $y = 1$ gives
\[
2x^2 - 15x - 7 = 0.
\]
Here $a = 2$, $b = -15$, $c = -7$. The discriminant is
\[
\Delta = b^2 - 4ac = (-15)^2 - 4 \cdot 2 \cdot (-7) = 225 + 56 = 281,
\]
which is not a perfect square. Hence no integer $x > 0$ satisfies the equation when $y = 1$.

\noindent\textbf{Case 4: $y = 2$.}  
Similarly, substituting $y = 2$ gives
\[
8x^2 - 25x - 12 = 0.
\]
Here $a = 8$, $b = -25$, $c = -12$. The discriminant is
\[
\Delta = b^2 - 4ac = (-25)^2 - 4 \cdot 8 \cdot (-12) = 625 + 384 = 1009,
\]
which is not a perfect square. Thus no integer $x > 0$ satisfies the equation when $y = 2$.

Since all possibilities for $0 < x < 3$ or $0 < y < 3$ lead to a quadratic with a non-square discriminant, there are no positive integer solutions $(x, y)$ to the given equation under these conditions. This completes the proof.
\end{proof}

\begin{lem}\label{lem:2}
The system of inequalities consisting of \eqref{embed-eq1} and \eqref{bound-eq1} has no integer solutions  with $x\ge 3$ and $y \ge 3$.
\end{lem}

\begin{proof} Let $x, y \ge 3$ be integers. Set $a = x - 1 \ge 2$, $b = y - 1 \ge 2$. The inequality becomes
$
ab < 6.
$
Since $a, b \ge 2$ are integers, the only possible pair $(a, b)$ with $ab < 6$ is $(a, b) = (2, 2)$. Hence $x = 3$ and $y = 3$.

Substituting $x = y = 3$ into the equation \eqref{embed-eq1} gives
$$
2 \cdot 3^2 \cdot 3^2 - 10 \cdot 3 \cdot 3 - 5 \cdot 3 - 5 \cdot 3 - 2
= 2 \cdot 81 - 90 - 15 - 15 - 2
= 162 - 122 = 40 \neq 0.
$$
Thus $(3, 3)$ does not satisfy the equation. Therefore no integers $x, y \ge 3$ satisfy both conditions simultaneously, completing the proof.
\end{proof}

\begin{lem}\label{lem:3}
There are no integer solutions to the system of inequalities \eqref{embed-eq2} with $x + y \le 1$, nor with $x - y \le 2$.
\end{lem}

\begin{proof} We consider two cases.

\noindent\textbf{Case 1: $x + y \le 1$.}  
Since $x > |y|$, we have $x + y > 0$. Combining with $x + y \le 1$ gives $x + y = 1$.  
Set $x = 1 - y$. Substituting into $x > |y|$ yields $1 - y > |y|$.

\begin{itemize}
\item If $y \ge 0$: $1 - y > y \Rightarrow 1 > 2y \Rightarrow y = 0$, so $x = 1$.  
Check the equation \eqref{embed-eq2}: \[x^2 - y^2 = 1, \quad LHS = (1-5)^2 = 16, \quad RHS = 5(3\cdot1+0)+29 = 44,\] not equal.

\item If $y < 0$: $1 - y > -y$ holds automatically. Substitute $x = 1-y$ into the equation \eqref{embed-eq2}:  
We have $x^2 - y^2 = (1-y)^2 - y^2 = 1 - 2y$.  
The equation becomes
\[
(1 - 2y - 5)^2 = 5[3(1-y) + y] + 29.
\]
Simplify:
\[
2y^2 + 13y - 14 = 0.
\]
Here $a = 2$, $b = 13$, $c = -14$. The discriminant is
\[
\Delta = b^2 - 4ac = 13^2 - 4 \cdot 2 \cdot (-14) = 169 + 112 = 281.
\]
Since $281$ is not a perfect square, the quadratic has no integer roots. Hence no integer $y$ satisfies the equation in this subcase.
\end{itemize}
Thus Case 1 yields no integer solutions.

\noindent\textbf{Case 2: $x - y \le 2$.}  
Since $x > |y|$, we have $x - y > 0$, so $x - y = 1$ or $2$.

\noindent\textit{Subcase 2.1: $x - y = 1$.} Then $x = y+1$, $x^2 - y^2 = 2y+1$.  
The equation becomes
\[
(2y + 1 - 5)^2 = 5[3(y+1) + y] + 29,
\]
Simplify:
\[
y^2 - 9y - 7 = 0.
\]
Here $a = 1$, $b = -9$, $c = -7$. The discriminant is
\[
\Delta = (-9)^2 - 4 \cdot 1 \cdot (-7) = 81 + 28 = 109.
\]
Since $109$ is not a perfect square, there is no integer $y$.

\noindent\textit{Subcase 2.2: $x - y = 2$.} Then $x = y+2$, $x^2 - y^2 = 4y+4 = 4(y+1)$.  
The equation becomes
\[
(4y + 4 - 5)^2 = 5[3(y+2) + y] + 29,
\]
Simplify:
\[
8y^2 - 14y - 29 = 0.
\]
Here $a = 8$, $b = -14$, $c = -29$. The discriminant is
\[
\Delta = (-14)^2 - 4 \cdot 8 \cdot (-29) = 196 + 928 = 1124.
\]
Since $1124 = 4 \cdot 281$ is not a perfect square, there is no integer $y$.
Hence Case 2 also yields no integer solutions.

Both cases lead to no integer solutions, completing the proof.
\end{proof}

\begin{lem}\label{lem:4}
The system of inequalities consisting of \eqref{embed-eq2} and \eqref{bound-eq2} has no integer solutions with $x + y > 1$ and $x - y > 2$.
\end{lem}

\begin{proof} The integer solutions of the inequality \eqref{bound-eq2} are
\[
(3,-1),\ (4,-2),\ (5,-3),\ (6,-4),\ (7,-5),\ (3,0),\ (4,-1),\ (5,-2),\ (4,0),\ (4,1),\ (5,1),\ (5,2),\ (6,3).
\]
Now we test each in the equation \eqref{embed-eq2}:
\begin{align*}
(3,-1)&:\ \text{left:}\ (9-1-5)^2=3^2=9,\qquad \text{right:}\ 5(9-1)+29=40+29=69;\\
(4,-2)&:\ \text{left:}\ (16-4-5)^2=7^2=49,\qquad \text{right:}\ 5(12-2)+29=50+29=79;\\
(5,-3)&:\ \text{left:}\ (25-9-5)^2=11^2=121,\quad \text{right:}\ 5(15-3)+29=60+29=89;\\
(6,-4)&:\ \text{left:}\ (36-16-5)^2=15^2=225,\quad \text{right:}\ 5(18-4)+29=70+29=99;\\
(7,-5)&:\ \text{left:}\ (49-25-5)^2=19^2=361,\quad \text{right:}\ 5(21-5)+29=80+29=109;\\
(3,0)&:\ \text{left:}\ (9-0-5)^2=4^2=16,\quad \text{right:}\ 5(9+0)+29=45+29=74;\\
(4,-1)&:\ \text{left:}\ (16-1-5)^2=10^2=100,\quad \text{right:}\ 5(12-1)+29=55+29=84;\\
(5,-2)&:\ \text{left:}\ (25-4-5)^2=16^2=256,\quad \text{right:}\ 5(15-2)+29=65+29=94;\\
(4,0)&:\ \text{left:}\ (16-0-5)^2=11^2=121,\quad \text{right:}\ 5(12+0)+29=60+29=89;\\
(4,1)&:\ \text{left:}\ (16-1-5)^2=10^2=100,\quad \text{right:}\ 5(12+1)+29=65+29=94;\\
(5,1)&:\ \text{left:}\ (25-1-5)^2=19^2=361,\quad \text{right:}\ 5(15+1)+29=80+29=109;\\
(5,2)&:\ \text{left:}\ (25-4-5)^2=16^2=256,\quad \text{right:}\ 5(15+2)+29=85+29=114;\\
(6,3)&:\ \text{left:}\ (36-9-5)^2=22^2=484,\quad \text{right:}\ 5(18+3)+29=105+29=134.
\end{align*}
None satisfy the equation. Therefore the system has no integer solutions.
\end{proof}



\begin{thebibliography}{100}


\bibitem{AL11}
M. \ Artebani,  A. \ Laface,
 {\it Cox rings of surfaces and the anticanonical Iitaka dimension,}
 Adv. Math. 
226 (6) (2011) 5252–5267.





\bibitem{BBCRDHJK12}
T.\ Bauer, C.\ Bocci, S.\ Cooper, S. D.\ Rocci, M. \ Dumnicki, B.\ Harbourne, K.\ Jabbusch, A. L.\ Knutsen, A.
\ K\"{u}ronya, R.\ Miranda, J.\ Ro\'{e}, H.\ Schenck, T.\ Szemberg, and Z.\ Teithler,
{\it Recent developments and open
problems in linear series},
In: Contributions to Algebraic Geometry, p. 93-140, EMS Ser. Congr. Rep.,
Eur. Math. Soc., Zürich (2012).




\bibitem{BHPV04}
W.\ Barth, K.\ Hulek, C.\ Peters, A.\ van de Ven,
{\it Compact Complex Surfaces}, 2nd ed.,
Ergeb. Math. Grenzgeb. {\bf 4}, Springer, Berlin (2004).


\bibitem{BCP08}
I.\ Bauer, F.\ Catanese, and R.\ Pignatelli,
{\it The classification of surfaces with $p_{g}=q=0$ isogenous to a product of curves},
in: Special Issue: In Honor of Fedor Bogomolov. Part 1, Pure Appl. Math. Q. {\bf 4} (2008), 547--586.

\bibitem{BCP11}
I.\ Bauer, F.\ Catanese, and R.\ Pignatelli,
{\it Surfaces of general type with geometric genus zero: a survey},
In: Ebeling, W., Hulek, K., Smoczyk, K. (eds) Complex and Differential Geometry, p1--48.
Springer Proceedings in Mathematics, vol {\bf 8}. Springer, Berlin, Heidelberg, 2011.

\bibitem{CCCK25}
P.\ Cascini, F.\ Catanese, Y.F.\ Chen and J.\ Keum
{\it Mori dream spaces and $\mathbb Q$-homology quadrics},
Pac.\ J.\ Math., vol. 339(2), 223--242, 2025.


\bibitem{C25}
F.\ Catanese
{\it Odd fake $\mathbb Q$-homology quadrics exist},
arXiv:2504.17475v1.


\bibitem{Dza14}
A.\ D\v{z}ambi\'{c},
{\it Fake quadrics from irreducible lattices acting on the product of upper half planes},
Math. Ann. {\bf 360} (2014), 23--51.


\bibitem{DR14}
A.\ D\v{z}ambi\'{c} and X.\ Roulleau,
{\it Automorphisms and quotients of quaternionic fake quadrics},
Pacific J. Math.  {\bf 267} (2014),  91--120.





\bibitem{FL22}
D.\ Frapporti and K.-S.\ Lee,
{\it Divisors on surfaces isogenous to a product of mixed type with $p_{g}=0$},
Pacific J. of Math. {\bf 318} (2022), 233--247.


\bibitem{FQ90}
M.~H. Freedman and F.~Quinn,
{\it Topology of 4-Manifolds},
Princeton Mathematical Series, Vol.~39, Princeton University Press, Princeton, NJ, 1990.




\bibitem{Har77}
R.\ Hartshorne,
{\it Algebraic Geometry},
Graduate Texts in Mathematics {\bf 52}, Springer-Verlag, New York, (1977).

\bibitem{Hir87}
F.\ Hirzebruch,
{\it Gesammelte Abhandlungen}, vol. {\bf I}. Springer, Berlin (1987).


\bibitem{KL19}
J.\ Keum and K.-S.\ Lee
{\it Examples of Mori dream surfaces of general type with $p_{g}=0$},
Adv. Math. {\bf 347} (2019), 708--738.


\bibitem{Lan01}
A.\ Langer,
{\it The Bogomolov-Miyaoka-Yau inequality for log canonical surfaces},
J.\ London Math.\ Soc.\ {\bf 64} (2001), 327--343.




\bibitem{LA04}
R.\ Lazarsfeld,
{\it Positivity in algebraic geometry,I.},
Ergeb. Math. Grenzgeb. {\bf 3},  Springer-Verlag, Berlin, Heidelberg, New York (2004).









\bibitem{LM95}
S.S.-Y.\ Lu  and  Y.\ Miyaoka,
{\it Bounding curves in algebraic surfaces by genus and Chern numbers},
Math. Res. Lett. {\bf 2} (1995), 663--676.

\bibitem{LSV19}
B.\ Linowitz, M.\ Stover, and J.\ Voight,
{\it Commensurability classes of fake quadrics},
Sel. Math. New Ser. {\bf 25} (2019), 48.





\bibitem{Miy84}
Y.\ Miyaoka,
{\it The maximal number of quotient singularities on surfaces with given numerical invariants},
Math. Ann. {\bf 268} (1984), 159--171.



\bibitem{PY07}
G.\ Prasad and S.-K.\ Yeung,
{\it Fake projective planes},
Invent. Math. {\bf 168} (2007), 321--370.



\bibitem{Via17}
C.\ Vial,
{\it Exceptional collections, and the N\'{e}ron-Severi lattice for surfaces},
Adv. Math. {\bf 305} (2017), 895--934.




\bibitem{Yau77}
S.-T.\ Yau,
{\it Calabi's conjecture and some new results in algebraic geometry},
Proc. Natl. Acad. Sci. USA {\bf 74} (1977), 1798--1799.

\bibitem{Yau78}
S.~T.\ Yau,
{\it On the Ricci curvature of a compact K\"ahler manifold and the complex Monge-Amp\'ere equation, I},
Comm. Pure Appl. Math,  vol.  {\bf 31} (1978), 339--441.




\bibitem{Zak93}
F.L.\ Zak,
{\it Tangents and secants of algebraic varieties},
Translations of Mathematical Monographs, vol. {\bf 127},  American Mathematical Society, Providence, RI, (1993).




\end{thebibliography}
\end{document}